\documentclass[11pt]{article}
\usepackage{amssymb,latexsym,amsmath,amstext,amsthm,fullpage,tabls,graphicx} 
\newcommand{\B}{\ensuremath{\mathcal{B}}}
\newcommand{\LLL}{\ensuremath{\mathcal{L}}}
\newtheorem{theorem}{Theorem}%[section]

\newtheorem{corollary}[theorem]{Corollary}
\begin{document}

\title{A Problem Concerning Nonincident Points and Lines in Projective Planes}
\author{Douglas R.\ Stinson%
\thanks{Research supported by NSERC grant 203114-2011}
\\David R.\ Cheriton School of Computer Science\\
University of Waterloo\\
Waterloo, Ontario N2L 3G1, Canada
}
\date{\today}

\maketitle

\begin{abstract}In this paper, we study the problem of finding the largest possible
set of $s$ points and $s$ lines in a projective plane of order $q$, such that that none
of the $s$ points lie on any of the $s$ lines. We prove that $s \leq 1+(q+1)(\sqrt{q}-1)$.
We also show that equality can be attained in this bound whenever $q$ is an even power of two.
\end{abstract}

\section{Introduction}

Suppose $\Pi = (X, \LLL)$ is a projective plane of order $q$, where
$X$ is the set of points and $\LLL$ is the set of lines in $\Pi$.
For $Y \subseteq X$ and $\mathcal{M} \subseteq \LLL$, we say that
$(Y, \mathcal{M})$ is a {\it nonincident} set of points and lines
if $y \not\in M$ for every $y \in Y$ and every $M \in \mathcal{M}$.

Define $f(\Pi)$ to be the maximum integer $s$ such that there exists 
a nonincident set of $s$ points and $s$ lines in $\Pi$. Equivalently,
$f(\Pi)$ is the size of the largest square submatrix of zeroes in the incidence
matrix of $\Pi$.
We use a simple combinatorial argument to prove the upper bound $f(\Pi) \leq 
1+(q+1)(\sqrt{q}-1)$, which holds for any projective plane $\Pi$ of order $q$.
We also show that this bound is tight
in certain cases, namely, for the desarguesian plane PG$(2,q)$
when $q$ is an even power of two. This is done 
by utilising maximal arcs.

\section{Main Results}

\begin{theorem} 
\label{t1}
For any set $Y$ of $s$ points in a projective plane of order $q$,
the number of lines disjoint from $Y$ is at most 
\[  \frac{q^3+q^2+q-qs}{q+s} .\]
\end{theorem}

\begin{proof}
Suppose that $(X, \LLL)$ is a projective plane of order $q$.
For a subset $Y\subseteq X$ of $s$ points, define $\LLL_Y = \{L \in \LLL : L \cap Y \neq \emptyset\}$
and define $\LLL'_Y = \LLL \setminus \LLL_Y$. Furthermore, for every $L \in \LLL_Y$, 
define $L_Y = L \cap Y$, and then define $\B = \{L_Y: L \in \LLL_Y\}$. Observe that
$\B$ consists of the nonempty intersections of the lines in $\LLL$ with  the set $Y$.
Denote $b = |\B| = |\LLL_Y|$.

We will study the set system
$(Y,\B)$.  We have the following equations:
\begin{eqnarray*}
\sum_{B \in \B} 1 &=& b\\
\sum_{B \in \B} |B| &=& (q+1)s\\
\sum_{B \in \B} \binom{|B|}{2} &=& \binom{s}{2}.
\end{eqnarray*}
 From the above equations, it follows that 
\begin{eqnarray*}
\sum_{B \in \B} |B|^{2} &=& s(q+s).
\end{eqnarray*}
The $q+1$ blocks in $\B$ through any point $y \in Y$
contain the $s-1$ points in $Y \setminus \{y\}$ once each, as well as $q+1$ occurrences
of $y$. So the average size of a block in $\B$ that contains any given point $y\in Y$ is $(q+s)/(q+1)$.
Therefore we define $\beta = (q+s)/(q+1)$
and compute as follows:
\begin{eqnarray*}
0 &\leq& \sum_{B \in \B} (|B|-\beta)^{2} \\
&=& s(q+s) - 2\beta(q+1)s + \beta^2 b ,
\end{eqnarray*}
from which it follows that
\begin{eqnarray*} b &\geq& \frac{s(2\beta(q+1) - (q+s))}{\beta^2}\\
&=& \frac{(q+1)^2s}{q+s}.
\end{eqnarray*}
Therefore,
\begin{eqnarray*}
|\LLL'_Y| &=& q^2+q+1 - b \\
& \leq& q^2+q+1 - \frac{(q+1)^2s}{q+s} \\
&=& \frac{q^3+q^2+q-qs}{q+s}.
\end{eqnarray*}
\end{proof}

\noindent{\it Remark.}
The inequality $b \geq (q+1)^2s/(q+s)$ that we proved above is in fact
a well-known result that has been proven in many different
guises over the years. For example, Mullin and Vanstone \cite{MV} proved
that $b \geq r^2v / (r + \lambda(v-1))$ in any $(r,\lambda)$ design on $v$ points.
If we let $r = q+1$, $\lambda=1$ and $v=s$, then we obtain $b \geq (q+1)^2s/(q+s)$.

\begin{corollary}
If there exists a nonincident set of $s$ points and $t$ lines in a projective
plane of order $q$, then 
\[ t \leq \frac{q^3+q^2+q-qs}{q+s}.\]
\end{corollary}

Before proving our next general result, we look at a small example. 
In Figure \ref{plot1.fig}, we graph the functions
$(q^3+q^2+q-qs)/(q+s)$ and $s$ for $q=16$ and $s\leq 100$. 
The point of intersection is $(52,52)$ and it is then easy to see that
 $f(\Pi) \leq 52$ for any projective plane $\Pi$ of order $16$.

\begin{figure}
\caption{Nonincident points and lines when $q=16$}
\label{plot1.fig}
\begin{center}
\includegraphics[width=90mm]{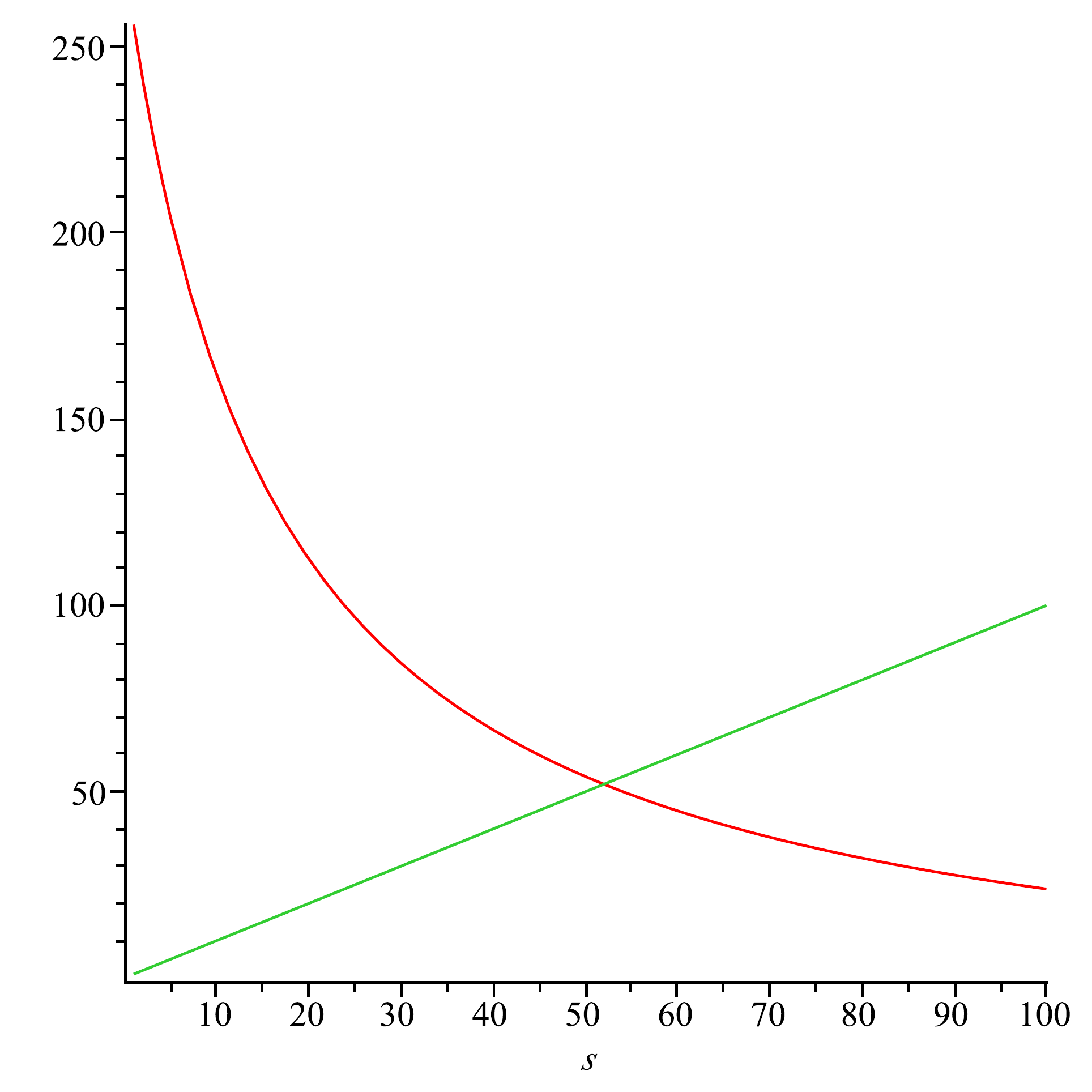}
\end{center}
\end{figure}

In general, it is easy to compute the  point of intersection of these two functions 
as follows:
\begin{eqnarray*} \frac{q^3+q^2+q-qs}{q+s} = s &\Leftrightarrow & 
s^2 +2qs- (q^3+q^2+q) = 0
\\
&\Leftrightarrow & s = -q \pm \sqrt{q^3+2q^2+q}\\
&\Leftrightarrow & s = -q \pm (q+1)\sqrt{q}.
\end{eqnarray*}
Since $s > 0$, the point of intersection occurs
when \[s= -q + (q+1)\sqrt{q} = 1+ (q+1)(\sqrt{q}-1).\]
The following result is now straightforward.

\begin{theorem} 
\label{main.thm} For any projective plane $\Pi$ of order $q$,
it holds that $f(\Pi) \leq 1+ (q+1)(\sqrt{q}-1)$.
\end{theorem} 

\begin{proof} Suppose there is a nonincident set of $s$ points and
$s$ lines in a projective plane of order $q$. Theorem \ref{t1} 
implies that $s \leq (q^3+q^2+q-qs)/(q+s)$. However, for $s > 1+ (q+1)(\sqrt{q}-1)$,
we have that $s > (q^3+q^2+q-qs)/(q+s)$, just as in the example considered above.
It follows that $s \leq 1+ (q+1)(\sqrt{q}-1)$.
\end{proof}

Next, we examine the case of equality in Theorem \ref{t1}. 
This will involve maximal arcs, which we now define.
A \emph{maximal $(s,\beta)$-arc} in a projective plane of order $q$ is a set $Y$ of
$s$ points such that every line meets $Y$ in $0$ or $\beta$ points. 
It is well-known that a maximal $(s,\beta)$-arc has $s = 1+(q+1)(\beta-1)$ and 
the number of lines that intersect the maximal arc is precisely $s(q+1)/
\beta = s(q+1)^2/(q+s)$.
For additional information on maximal arcs, see \cite[\S VI.41.3]{CD}.

\begin{corollary}
\label{c3}
Suppose we have a set $Y$ of $s$ points in a projective plane of order $q$ such that 
the number of lines disjoint from $Y$ is equal to
\[  \frac{q^3+q^2+q-qs}{q+s} .\]
Then $Y$ is a maximal $(s,\beta)$-arc, where $s = (q+1)(\beta - 1) - 1$.
Conversely, if $Y$ is a maximal  $(s,\beta)$-arc in a projective plane of order $q$, then 
number of lines disjoint from $Y$ is equal to
$(q^3+q^2+q-qs)/(q+s)$.
\end{corollary}
\begin{proof}
 From the proof of Theorem \ref{t1}, it is easy to see that equality holds if and
only if every line in $\LLL_Y$ meets $Y$ in exactly $\beta$ points, where
$\beta = (q+s)/(q+1)$. It immediately follows that every line in the plane meets $Y$ in $0$ or $\beta$ points,
and therefore $Y$ is a maximal arc. 
The converse follows from the basic properties of maximal arcs mentioned above.
\end{proof}

\begin{theorem}
\label{even.thm}
When $q$ is an even power of $2$, 
there exist nonincident sets of $s$ 
points and 
$s$ lines in 
PG$(2,q)$, where $s = 1+(q+1)(\sqrt{q}-1)$. 
\end{theorem}

\begin{proof}
Denniston \cite{denn} proved that there is a maximal $(s,2^u)$-arc
in PG$(2,2^v)$ whenever $0 < u < v$. Suppose $v$ is even and we take
$u = v/2$. Therefore we have a maximal $(s,\beta)$-arc, where $q = 2^v$, 
$\beta = \sqrt{q}$ and $s = 1+(q+1)(\sqrt{q}-1)$.
 
Suppose we take $Y$ to be the $s$ points in the arc and we apply Corollary \ref{c3}.
Since $s = 1+(q+1)(\beta-1)$, there are exactly $s$ lines in $\LLL'_Y$.
Therefore we have a nonincident set of $s$ points and $s$ lines in PG$(2,q)$.
\end{proof}

\begin{corollary}
\label{even.cor}
If $q$ is an even power of $2$, then  $f(\text{PG}(2,q)) = 1+(q+1)(\sqrt{q}-1)$.
\end{corollary}

\begin{proof}
 From Theorem \ref{even.thm}, we have $f(\text{PG}(2,q)) \geq 1+(q+1)(\sqrt{q}-1)$.
However, $f(\text{PG}(2,q)) \leq 1+(q+1)(\sqrt{q}-1)$ from Theorem \ref{main.thm}.
It follows that $f(\text{PG}(2,q)) = 1+(q+1)(\sqrt{q}-1)$.
\end{proof}

In the case where $q$ is odd, it was shown by
Ball, Blokhuis and Mazzocca \cite{BBM} that there is no nontrivial maximal
arc in the desarguesian plane PG$(2,q)$. The existence of maximal arcs
for nondesarguesian projective planes of odd order is unresolved at
the present time.

\section*{Acknowledgements}
Thanks to Douglas West for posing the problem studied in this paper, and for
making useful comments on earlier drafts of this paper.

\end{document}